\documentclass{amsart}
\usepackage{amssymb,url,xspace, amsthm}
\usepackage{mathtools}
\usepackage{bigints}
\usepackage{multirow}
\usepackage{hyperref}
\usepackage[all]{xy}

\newtheorem{theorem}{Theorem}[section]
\newtheorem{lemma}[theorem]{Lemma}
\newtheorem{props}[theorem]{Proposition}

\theoremstyle{definition}
\newtheorem{definition}[theorem]{Definition}

\newtheorem*{theorem*}{Theorem}

\theoremstyle{remark}

\numberwithin{equation}{section}

\newcommand{\N}{\mathbb{N}}
\newcommand{\E}{\mathcal{E}}

\newcommand{\T}{\mathbb{T}}
\newcommand{\C}{\mathbb{C}}
\newcommand{\D}{\mathbb{D}}
\newcommand{\Q}{\mathbb{Q}}

\DeclareMathOperator*{\esssup}{ess\,sup}

\begin{document}

\title{Complex interpolation of families of Orlicz sequence spaces}

\author{Willian Hans Goes Corr\^ea}
\address{Departamento de Matem\'atica, Instituto de Matem\'atica e Estat\'istica, Universidade de S\~ao Paulo, Rua do Mat\~ao 1010, 05508-090 S\~ao Paulo SP, Brazil}
\email{willhans@ime.usp.br}
\thanks{The present work was supported by S\~ao Paulo Research Foundation (FAPESP), processes 2016/25574-8 and 2018/03765-1.}



\date{}

\dedicatory{}

\begin{abstract}
We study the complex interpolation and derivation process induced by a
family of Orlicz sequence spaces. We present a concrete example of interpolation family of three spaces inducing a centralizer that cannot be obtained from complex interpolation of two spaces.
\end{abstract}

\maketitle

\section{Introduction}
Complex interpolation induces a homogeneous (usually nonlinear) map called the \emph{derivation map} (see the next section for the definition). Given a K\"othe function space $X$ on a Polish space $S$ with measure $\mu$, a \emph{centralizer} is a homogeneous map $\Omega : X \rightarrow L_0(\mu)$ for which there is a constant $C > 0$ such that for every $x \in X$, $u \in L_{\infty}(\mu)$ we have
\[
\|\Omega(ux) - u\Omega(x)\|_X \leq C \|u\|_{\infty} \|x\|_X
\]

In \cite{KaltonNonlinear, KaltonKothe} Kalton showed that the derivations induced by complex interpolation of families of K\"othe function spaces are centralizers. Conversely, given a centralizer $\Omega$ on a superreflexive K\"othe function space $X$ there is an interpolation family of K\"othe function spaces such that the interpolation space at 0 is $X$ and the induced derivation is $\Omega$. If the centralizer is real (i.\ e., $\Omega(f)$ is a real function if $f$ is real) then a family of two spaces is enough to induce $\Omega$; otherwise, three spaces are enough.

Kalton's results leave unanswered the question of whether three spaces are necessary to obtain any centralizer. The answer is yes (see Section \ref{sec:preliminary}). The aim of this work is to give a concrete example of family of three K\"othe sequence spaces inducing a centralizer not induced from interpolation of two spaces.

In the search of an example one is tempted to use interpolation scales of $\ell_p$ spaces or of $\ell_{p, q}$ spaces. However, the reiteration result of \cite{Stability} shows that this is bound to fail. A natural candidate then is a family formed by Orlicz spaces. In Section \ref{sec:interOrlicz} we present a detailed treatment of complex interpolation of families of Orlicz sequence spaces and their associated derivations. Finally, in Section \ref{sec:example} we obtain the desired example.

\section{Complex interpolation and derivations}

We present our results in the context of Kalton's interpolation method for K\"othe function spaces \cite{KaltonKothe}. A comparison with the classical method (\cite{Coifman1982, Hernandez}) can be seen at \cite{Stability}. In what follows $\C^{\N}$ is endowed with the product topology.

Kalton's definition of K\"othe function space is not the classic one. A K\"othe function space $X$ on $\N$ is a linear subspace of $\C^{\N}$ with a norm $\|.\|_X$ which makes it into a complete Banach space such that (for $x \in \C^{\N} \setminus X$ we make the convention that $\|x\|_X = \infty$):
\begin{enumerate}
    \item $B_X$ is closed in $\C^{\N}$, where $B_X$ is the closed unit ball of $X$;
    \item For every $x, y \in \C^{\N}$, if $y \in X$ and $\left|x\right| \leq \left|y\right|$ then $x \in X$ and $\|x\|_X \leq \|y\|_X$;
    \item There are $h, k \in \C^{\N}$ strictly positive such that for every $x \in \C^{\N}$ we have $\|xh\|_1 \leq \|x\|_X \leq \|xk\|_{\infty}$.
\end{enumerate}

As a comparison, in the standard definition one asks that the characteristic function of every finite set is in $X$. Kalton's definition implies that $\chi_{\{n\}} \in X$ for every $n \in \N$, since by (3) there is a positive element in $X$.

Let us consider the unit circle $\T$ together with normalized Haar measure $\lambda$.

\begin{definition}[Strongly admissible family]
A \emph{strongly admissible family} is a family $\mathcal{H} = \{X_w\}_{w \in \mathbb{T}}$ of K\"othe function spaces on $\N$ such that
\begin{enumerate}
    \item The map $(w, x) \mapsto \|x\|_w$ from $\T \times \C^{\N}$ into $[0, \infty]$ is Borel;
    \item There are $h, k \in \C^{\N}$ strictly positive such that for every $x \in \C^{\N}$ and a.\ e.\ $w \in \T$ we have $\|xh\|_1 \leq \|x\|_w \leq \|xk\|_{\infty}$;
    \item There is a linear subspace $V$ of $\C^{\N}$ of countable dimension such that $V \cap B_{X_w}$ is $\C^{\N}$-dense in $B_{X_w}$ for a.\ e.\ $w \in \T$.
\end{enumerate}
A family that satisfies only (1) and (2) is called \emph{admissible}.
\end{definition}

We will be interested in finite families:

\begin{definition}
An admissible family $\mathcal{H} = \{X_w\}_{w \in \T}$ is \emph{finite} if there are $n \in \N$, a family $\{X^j : 1 \leq j \leq n\}$ of K\"othe function spaces and a function $f : \T \rightarrow \{1, ..., n\}$ such that $X_w = X^{f(w)}$ for every $w \in \T$ and $f^{-1}(j)$ is an arc for $1 \leq j \leq n$. If we denote $A_j = f^{-1}(j)$, then we also write $\mathcal{H} = \{X^j; A_j\}$. If $X^j \neq X^k$ for $1 \leq j < k \leq n$, we say that $\mathcal{H}$ is a \emph{family of $n$ spaces}.
\end{definition}

Let $\D$ denote the open unit disk.

\begin{definition}
Let $\mathcal{H} = \{X_w\}_{w \in \T}$ be an admissible family. The space $\mathcal{N}^+ = \mathcal{N}^+(\mathcal{H})$ is the space of functions $F : \D \rightarrow \C^{\N}$ such that $F(.)(n) : \D \rightarrow \C$ is in the Smirnov class $N^+$ for every $n \in \N$ and
\[
\|F\|_{\mathcal{N}^+} = \esssup_{w \in \T} \|F(w)\|_{X_w} < \infty
\]
where $F(w)$ is the radial limit of $F$ at $w$ in $\C^{\N}$. For $z \in \D$ the K\"othe sequence space $X_z$ is
\[
X_z = \{F(z) : F \in \mathcal{N}^+\}
\]
with the quotient norm. Let $B : X_z \rightarrow \mathcal{N}^+$ be a homogeneous function for which there is $C \geq 1$ such that $\|B(x)\|_{\mathcal{N}^+} \leq C \|x\|_{X_z}$ and $B(x)(z) = x$ for every $x \in X_z$. The \emph{derivation} $\Omega_z : X_z \rightarrow \C^{\N}$ is defined as $\Omega_z(x) = B(x)'(z)$.
\end{definition}

In most cases, Kalton's method of interpolation agrees with the classical complex method for families from \cite{Coifman1982} (see \cite{Stability}). Kalton showed that the derivations induced by complex interpolation of K\"othe function spaces are centralizers (see the definition in the Introduction). Notice that the exact expression of $\Omega_z$ depends on the choice of $B$, but another choice induces the same centralizer up to bounded equivalence:

\begin{definition}
Let $X$ be a K\"othe function space on $\N$. Two centralizers $\Omega$ and $\Psi$ on $X$ are called \emph{equivalent} if there is a linear map $L : X \rightarrow \C^{\N}$ such that $\Omega - \Psi - L : X \rightarrow X$ is bounded, in the sense that there is a constant $C > 0$ such that
\[
\|\Omega(x) - \Psi(x) - L(x)\|_{X} \leq C \|x\|_X
\]
for every $x \in X$. If we may take $L = 0$ then $\Omega$ and $\Psi$ are said \emph{boundedly equivalent}, and we write $\Omega \sim \Psi$. If there is $\lambda \in \C$, $\lambda \neq 0$ such that $\Omega$ is equivalent to $\lambda \Psi$ then $\Omega$ and $\Psi$ are said \emph{projectively equivalent}. If $\Omega$ is boundedly equivalent to $\lambda \Psi$, then $\Omega$ and $\Psi$ are \emph{boundedly projectively equivalent}. $\Omega$ is said to be \emph{trivial} if it is equivalent to 0.
\end{definition}

We shall use an alternative description of $X_z$ that is akin to Lovanozkii factorization.

\begin{definition}
Let $\mathcal{H} = \{X_w\}_{w \in \mathbb{T}}$ be an admissible family. The space $\E = \E(\mathcal{H})$ is the space of all functions (up to almost everywhere equivalence) $\phi : \N \times \T \rightarrow \C$ such that
\[
\|\phi\|_\E = \esssup \|\phi(., w)\|_{X_w} < \infty
\]
\end{definition}

Let $P(r, t)$ denote the Poisson kernel, that is, $P(r, t) = \frac{1 - r^2}{1 - 2r \cos(t) + r^2}$.

\begin{definition}[\cite{KaltonKothe}, Theorem 3.3]
Let $\mathcal{H} = \{X_w\}_{w \in \mathbb{T}}$ be an admissible family. Then $X_z$ coincides with the space of all $x \in \C^{\N}$ such that there is $\phi \in \E$, $\phi \geq 0$, with
\begin{equation}\label{eq:def-xz}
\left|x(n)\right| = \exp(\frac{1}{2\pi}\int_{-\pi}^{\pi} P(r, \theta - t) \log \phi(n, e^{it}) dt)
\end{equation}
for every $n \in \N$. Furthermore, $\|x\|_{X_z} = \inf \|\phi\|_\E$, where the infimum is over all $\phi$ satisfying \eqref{eq:def-xz}.
\end{definition}

We remark that \cite[Theorem 3.3]{KaltonKothe} is stated only for strongly admissible families, but the result is valid for admissible families in general. We shall also use an alternative description for the derivation $\Omega_z$ that may be easily derived from the proof of \cite[Theorem 3.3]{KaltonKothe}.

\begin{definition}\label{def:inducedcent}
Let $\mathcal{H} = \{X_w\}_{w \in \T}$ be an admissible family and $z \in \D$. A map $B_z : X_z \rightarrow \E$ is a \emph{factorization map for $\mathcal{H}$ at $z$} if there is $C > 0$ such that for every $x \in X_z$ and $\lambda \in \C$
\begin{enumerate}
    \item $B_z(x) \geq 0$;
    \item $B_z(\lambda x) = \left|\lambda\right| B_z(x)$
    \item $\left|x(n)\right| = \exp(\frac{1}{2\pi}\int_{-\pi}^{\pi} P(r, \theta - t) \log B_z(n, e^{it}) dt)$;
    \item $\|B_z(x)\|_{\E} \leq C \|x\|_{X_z}$.
\end{enumerate}
\end{definition}

\begin{props}\label{props:inducedcent}
Let $\mathcal{H}$ be an admissible family and $B_z$ a factorization map for $\mathcal{H}$ at $z$. Then
\[
\Omega_z(x)(n) = \frac{x(n)}{\pi} \int_{-\pi}^{\pi} \frac{e^{it}}{(e^{it} - z)^2} \log B_z(x)(n, e^{it}) dt
\]
\end{props}

\begin{definition}
A centralizer $\Omega$ on a K\"othe function space $X$ is \emph{real} if $\Omega(f)$ is a real function for every real function $f \in X$.
\end{definition}

In \cite{KaltonNonlinear, KaltonKothe} Kalton establishes the following:

\begin{theorem}\label{thm:Kalton}
\emph{(1)}   Let $\mathcal{H}$ be a family of two K\"othe function spaces and $z \in \D$. Then $\Omega_z$ is boundedly projectively equivalent to a real centralizer.

\emph{(2)} If $\Omega$ is a (real) centralizer on a superreflexive K\"othe function space $X$ then there is a family of three (two) K\"othe function spaces such that $X_0 = X$ and $\Omega_0$ is boundedly projectively equivalent to $\Omega$.

\end{theorem}

The \emph{K\"othe dual} $X'$ of a K\"othe space $X$ is defined as the space of all $y \in \C^{\N}$ such that
\[
\|y\|_{X'} = \sup_{x \in B_X} \sum_{n=1}^{\infty} \left|y(n)x(n)\right| < \infty
\]

Then $X'$ is a subspace of $X^*$. We have the following duality result:
\begin{theorem}\cite[Lemma 3.2 and Corollary 4.8]{KaltonKothe}\label{thm-duality}
Let $\mathcal{H} = \{X_w\}_{w \in \T}$ be a strongly admissible family. Then the family $\mathcal{H}' = \{Y_w\}_{w \in \T}$ where $Y_w = X_w'$ is admissible and $Y_z = X_z'$ isometrically.
\end{theorem}

\section{An abstract counterexample}\label{sec:preliminary}

In this section we show that Kalton's theorem \ref{thm:Kalton} is optimal, in the sense that two spaces are not enough to induce a given centralizer.

If $X$ is a K\"othe function space and $\Omega$ is a centralizer on $X$ then there are real centralizers $\Omega_1, \Omega_2$ on $X$ such that $\Omega \sim \Omega_1 + i \Omega_2$ \cite[Lemma 7.1]{KaltonKothe}. The following result characterizes when $\Omega_1 + i \Omega_2$ is a real centralizer. If $f : \N \rightarrow \C$, we let $M_f : \C^{\N} \rightarrow \C^{\N}$ be given by $M_f(x)(n) = f(n)x(n)$.

\begin{lemma}\label{lem:realcent}
Let $\Omega \sim \Omega_1 + i\Omega_2$ be a centralizer on $X$ with $\Omega_1$, $\Omega_2$ real centralizers. Then $\Omega$ is projectively equivalent to a real centralizer on $X$ if and only if either
\begin{enumerate}
    \item $\Omega_1$ or $\Omega_2$ is trivial;
    \item $\Omega_1$ and $\Omega_2$ are projectively equivalent.
\end{enumerate}
\end{lemma}
\begin{proof}
Let $\Psi$ be a real centralizer on $X$, and suppose that $\Omega \sim (\alpha + i\beta) \Psi$, with $\alpha, \beta \in \mathbb{R}$. That is, there are $f_1, f_2 : \N \rightarrow \mathbb{R}$ such that
\[
\Omega_1 + i\Omega_2 - (\alpha + i\beta)\Psi - M_{f_1} -iM_{f_2}
\]
is bounded (and complex homogeneous).

This happens if and only if $\Omega_1 - \alpha \Psi - M_{f_1} = B_1$
and $\Omega_2 - \beta \Psi - M_{f_2} = B_2$ are bounded. Now, if $\alpha = 0$ then $\Omega_1$ is trivial. Otherwise $\Psi = \frac{1}{\alpha}(\Omega_1 - M_{f_1} - B_1)$, and substituting we get
    \[
    \Omega_2 - \frac{\beta}{\alpha} \Omega_1 = B_2 + M_{f_2} - \frac{\beta}{\alpha}(M_{f_1} + B_1)
    \]
that is, $\Omega_2$ is equivalent to $\frac{\beta}{\alpha}\Omega_1$. If $\beta = 0$ then $\Omega_2$ is trivial. Otherwise $\Omega_1$ and $\Omega_2$ are projectively equivalent. The sufficiency is easy.
\end{proof}

F\'elix Cabello S\'anchez presented the following result to us:
\begin{theorem}
There is a centralizer which is projectively equivalent to a centralizer induced by complex interpolation of a family of three K\"othe function spaces but non projectively equivalent to one induced by complex interpolation of two K\"othe function spaces.
\end{theorem}
\begin{proof}
Let $\Omega_1$ and $\Omega_2$ be two real centralizers on $\ell_2$, both nontrivial, which are not projectively equivalent. By Kalton's Theorem \ref{thm:Kalton} there is a family of three K\"othe function spaces which induces a centralizer projectively equivalent to $\Omega_1 + i\Omega_2$, which by Lemma \ref{lem:realcent} is not projectively equivalent to a centralizer induced by complex interpolation of a family of two K\"othe function spaces.
\end{proof}


\section{Complex interpolation of Orlicz spaces}\label{sec:interOrlicz}

For general facts on Orlicz spaces we refer to \cite[Chapter 4]{ClassicalI} and \cite{RaoRen}. Recall that $\phi : [0, \infty) \rightarrow [0, \infty)$ is an \emph{Orlicz function} if it is convex, non-decreasing, $\phi(0) = 0$ and $\lim_{t \rightarrow \infty} \phi(t) = \infty$. The function $\phi$ is said \emph{non-degenerate} if $\phi(t) > 0$ for $t > 0$, and satisfies the \emph{$\Delta_2$-condition at 0} if $\limsup_{t \rightarrow 0} \frac{\phi(2t)}{\phi(t)} < \infty$. We deal only with non-degenerate Orlicz functions, unless otherwise stated.

Given an Orlicz function $\phi$ we have the Orlicz space
\[
\ell_{\phi} = \{x \in \C^{\N} : \sum \phi\Big(\frac{\left|x(n)\right|}{\rho}\Big) < \infty \text{ for some } \rho > 0\}
\]
endowed with the complete norm
\[
\|x\|_{\phi} = \inf\{\rho > 0 : \sum \phi\Big(\frac{\left|x(n)\right|}{\rho}\Big) \leq 1\}
\]

The $\Delta_2$-condition at 0 is equivalent to the separability of $\ell_{\phi}$. Of course, the most famous examples of Orlicz spaces are the $\ell_p$ spaces.

An \emph{N-function} is a nondegenerate Orlicz function $\phi$ such that $\lim_{t \rightarrow 0} \frac{\phi(t)}{t} = 0$ and $\lim_{t \rightarrow \infty} \frac{\phi(t)}{t} = \infty$. For a nondegenerate Orlicz function $\phi$ which is not an $N$-function $\ell_{\phi} \simeq \ell_1$. If $\phi$ is an N-function satisfying the $\Delta_2$-condition at 0 then there is an N-function $\phi^*$ such that $\ell_{\phi}^* \simeq \ell_{\phi^*}$, where the action of an element $x^* \in \ell_{\phi^*}$ on $x \in \ell_{\phi}$ is given by $\sum_{n = 1}^{\infty} x^*(n) x(n)$. In fact,
\begin{equation}\label{eq:equivalencedual}
\|x^*\|_{\ell_{\phi^*}} \leq \|x^*\|_{\ell_{\phi}^*} \leq 2\|x^*\|_{\ell_{\phi^*}}
\end{equation}

The function $\phi^*$ is given by $\phi^*(y) = \sup\{xy - \phi(x) : 0 < x < \infty\}$. If $\phi$ satisfies the $\Delta_2$-condition at 0 then $\ell_{\phi}^* = \ell_{\phi}'$.

It is clear that Orlicz spaces are K\"othe function spaces. If $\phi, \psi$ are Orlicz functions, then $\ell_{\phi} = \ell_{\psi}$ with equivalence of norms if and only if there are $t_0, k, K > 0$ such that
\[
\frac{1}{K} \phi\Big(\frac{t}{k}\Big) \leq \psi(t) \leq K \phi(kt)
\]
for every $t \in [0, t_0]$. If $\phi$ or $\psi$ satisfies the $\Delta_2$-condition at zero then we can take $k = 1$.

\begin{definition}\label{def:admissibleOrliczfamily}
A family $\{\phi_w\}_{w \in \T}$ of Orlicz functions will be called \emph{admissible} if
\begin{enumerate}
    \item For every $t \in [0, \infty)$ the function $w \mapsto \phi_w(t)$ is Borel;
    \item There is $t_0 > 0$ such that for every $t \in (0, t_0)$ the function $w \mapsto \log \phi_w^{-1}(t)$ is integrable.
    \item There are $h, k \in \ell_0$ strictly positive such that for every $x \in \ell_0$ and a.\ e.\ $w \in \T$ we have $\|xh\|_1 \leq \|x\|_{\ell_{\phi_w}} \leq \|xk\|_{\infty}$.
\end{enumerate}

If, furthermore, the functions $\phi_w$ satisfy the $\Delta_2$-condition at 0, the family will be called \emph{strongly admissible}.
\end{definition}

\begin{props}\label{props:afamilyisadmissible}
Let $\{\phi_w\}_{w \in \T}$ be a (strongly) admissible family of Orlicz functions. Then $\mathcal{H} = \{X_{\ell_w}\}_{w \in \T}$ is a (strongly) admissible family.
\end{props}

\begin{proof}
We must show that the map $T : (w, x) \mapsto \|x\|_{\ell_{\phi_w}}$ is Borel.

Let $\psi_{\lambda} : \T \times \C^{\N} \rightarrow [0, \infty]$ be given by $\psi_{\lambda}(w, x) = \sum_{n=1}^{\infty} \phi_w\Big(\frac{\left|x(n)\right|}{\lambda}\Big)$ and $a \in [0, \infty]$. We want to check that $T^{-1}[0, a)$ is a Borel set. Let $(\lambda_n)$ be an enumeration of $\Q \cap [0, a)$. We have
\begin{eqnarray*}
(w, x) \in T^{-1}[0, a) & \iff & \inf\{\rho : \sum_{n=1}^{\infty} \phi_w\Big(\frac{\left|x(n)\right|}{\rho}\Big) \leq 1\} < a\\
    & \iff & \exists k : \sum_{n=1}^{\infty} \phi_w\Big(\frac{\left|x(n)\right|}{\lambda_k}\Big) \leq 1\\
    & \iff & (w, x) \in \bigcup\limits_{k=1}^{\infty} \psi_{\lambda_k}^{-1}[0, 1]
\end{eqnarray*}
so we just need to check that $\psi_{\lambda_k}^{-1}[0, 1]$ is a Borel set for every $k$.

Fix $\lambda$ and let $A_{k, n, m} = \{x \in \C^{\N} : \frac{\left|x(k)\right|}{\lambda} \in [\frac{n}{2^m}, \frac{n+1}{2^{m}})\}$.  This is a Borel set, since the function $x \mapsto \frac{\left|x(k)\right|}{\lambda}$ is continuous. Let $\xi_{k, n} : \T \times \C^{\N} \rightarrow [0, \infty]$ be given by
\[
\xi_{k, m}(w, x) = \sum_{n=0}^\infty \chi_{A_{k, n, m}}(x) \phi_w\Big(\frac{n}{2^m}\Big)
\]

The functions $(w, x) \mapsto \chi_{A_{k, n, m}}(x)$ and $(w, x) \mapsto \phi_w\Big(\frac{n}{2^m}\Big)$ are Borel, so that $\xi_{k, n}$ is Borel. But $\xi_{k, m}(w, x)$ converges in $m$ to $\phi_w\Big(\frac{\left|x(k)\right|}{\lambda}\Big)$, so that this function is Borel.

This implies that the function $\psi_{\lambda_k, N}(w, x) = \sum_{n=1}^N \phi_w\Big(\frac{\left|x(n)\right|}{\lambda_k}\Big)$ is Borel, and therefore so is $\psi_{\lambda_k}$, proving that $(w, x) \mapsto \|x\|_w$ is a Borel function.

For strong admissibility, we may take $V = c_{00}$. Indeed, take $x \in B_{X_w}$. Then the $N$-truncations $x^N \in V \cap B_{X_w}$ and converge to $x$ in $\C^{\N}$, since $x \in c_0$.
\end{proof}

\begin{props}
Condition (2) of Definition \ref{def:admissibleOrliczfamily} is equivalent to the existence of Orlicz functions $\phi, \psi$ such that $\phi \leq \phi_w \leq \psi$ on a neighborhood of $0$ for a.\ e.\ $w \in \T$. Here we allow $\phi$ to be degenerate.
\end{props}
\begin{proof}
Suppose $\phi$ and $\psi$ are as above. Then for a.\ e.\ $w \in \T$ we have continuous inclusions $\ell_{\psi} \subset \ell_{\phi_w} \subset \ell_{\phi}$, with uniform bound on their norms.

Let $y \in \ell_{\phi}'$ be strictly positive of norm at most 1. Then there is $a > 0$ such that $y \in \ell_{\phi_w}'$ with norm at most $a$ for a.\ e.\ $w \in \T$. Take $\frac{y}{a}$ as $h$.

Now take $k$ strictly positive such that $\sum \psi(k(n)^{-1}) \leq 1$ with $k(n)$ big enough so that $\phi_w(k(n)^{-1}) \leq \psi(k(n)^{-1})$ for every $n$ and a.\ e.\ $w \in \T$. Then for a.\ e.\ $w \in \T$ and every $x \in \C^{\N}$ we have
\begin{eqnarray*}
\sum \phi_w\Big(\frac{\left|x(n)\right|}{\|xk\|_{\infty}}\Big) & \leq & \sum \phi_w(k(n)^{-1}) \\
            & \leq & \sum \psi(k(n)^{-1}) \\
            & \leq & 1
\end{eqnarray*}

For the converse, suppose (2). Taking $x = e_1$ we have for a.\ e.\ $w \in \T$
\[
h(1) \leq \inf\{\rho > 0 : \phi_w\Big(\frac{1}{\rho}\Big) \leq 1\} = \frac{1}{\phi_w^{-1}(1)} \leq k(1)
\]
So $k(1)^{-1} \leq \phi_w^{-1}(1) \leq h(1)^{-1}$. Now we may take as $\phi$ a degenerate function which is 0 on $[0, h(1)^{-1}]$, and $\psi(t) = k(1)t$.
\end{proof}

Given an admissible family of Orlicz functions $\{\phi_w\}_{w \in \T}$ and $z = re^{i\theta} \in \D$ we denote
\[
I_z(\{\phi_w\})(t) = \exp\Big(\frac{1}{2\pi} \int_{-\pi}^{\pi} P(r, \theta - t) \log \phi_{e^{it}}^{-1}(t) dt\Big)
\]

\begin{props}\label{props:concave}
Let $\{\phi_w\}_{w \in \T}$ be an admissible family of Orlicz functions. Then for every $z \in \D$
\begin{enumerate}
    \item $I_z(\{\phi_w\})(0) = 0$;
    \item If $t_0 > 0$ is such that $w \mapsto \log \phi_w^{-1}(t)$ is integrable for every $t \in (0, t_0)$, then $I_z(\{\phi_w\})$ is concave and strictly increasing on $[0, t)$ for some $t < t_0$.
\end{enumerate}
\end{props}
\begin{proof}
\emph{(1)} Follows from $\phi_w(0) = 0$ for every $w \in \T$.

\emph{(2)} Let $t_0, t_1 \in [0, \infty)$ and $\alpha \in [0, 1]$. First, suppose that $\mu$ is a probability measure on $\T$ and let $f, g : \T \rightarrow (0, \infty)$ be simple functions, let us write $f(w) = \sum a_k \chi_{E_k}$ and $g = \sum b_k \chi_{E_k}$. Then for $\alpha \in [0, 1]$
\begin{eqnarray*}
&& \alpha \exp\Big(\int_{\T} \log f(w) dw\Big) + (1 - \alpha) \exp\Big(\int_{\T} \log g(w) dw\Big)\\
    & = & \alpha \prod a_k^{\mu(E_k)} + (1 - \alpha)  \prod b_k^{\mu(E_k)} \\
    & = & \prod (\alpha a_k)^{\mu(E_k)} + \prod ((1- \alpha) b_k)^{\mu(E_k)} \\
    & \leq & \prod_k [\alpha a_k + (1 - \alpha) b_k]^{\mu(E_k)} \\
    & = & \exp\Big(\log \prod_k [\alpha a_k + (1 - \alpha) b_k]^{\mu(E_k)}\Big) \\
    & = & \exp\Big(\int_{\T} \log[\alpha f(w) + (1 - \alpha) g(w)] dw\Big)
\end{eqnarray*}

Now let $f_n$ (respectively, $g_n$) be a monotone sequence of simple positive functions such that $f_n(w) \rightarrow \phi_w^{-1}(t_0)$ (respectively, $g_n(w) \rightarrow \phi_w^{-1}(t_1)$) for a.\ e.\ $w \in \T$. Then monotone convergence gives 
\begin{eqnarray*}
&& (I_z(\{\phi_w\})(\alpha t_0 + (1 - \alpha) t_1))  \\
& = & \exp\Big(\frac{1}{2\pi} \int_{0}^{2\pi} \log \phi_w^{-1}(\alpha t_0 + (1 - \alpha) t_1) P_z(w) dw\Big) \\
        & \geq & \exp\Big(\frac{1}{2\pi} \int_{0}^{2\pi} \log (\alpha \phi_w^{-1}(t_0) + (1 - \alpha)\phi_w^{-1}(t_1)) P_z(w) dw\Big) \\
        & = & \lim_n \exp\Big(\frac{1}{2\pi} \int_{0}^{2\pi} \log (\alpha f_n(w) + (1 - \alpha)g_n(w)) P_z(w) dw\Big) \\
        & \geq & \lim_n \alpha \exp\Big(\int_0^{2\pi} \log f(w) P_z(w) dw\Big) + (1 - \alpha) \exp\Big(\int_0^{2\pi} \log g(w) P_z(w) dw\Big) \\
        & = & \alpha \exp\Big(\int_0^{2\pi} \log \phi_w^{-1}(t_0) P_z(w) dw\Big) + (1 - \alpha) \exp\Big(\int_0^{2\pi} \log \phi_w^{-1}(t_1) P_z(w) dw\Big) \\
        & = & \alpha (I_z(\{\phi_w\})(t_0) + (1 - \alpha) (I_z(\{\phi_w\})(t_1)
\end{eqnarray*}

Since $I_z(\{\phi_w\})$ is concave and nonnegative, the only way for (2) to not be satisfied is if there is $0 < t' < t_0$ such that $I_z(\{\phi_w\})(t) = 0$ for $0 \leq t < t'$, which would imply that $w \mapsto \log \phi_w^{-1}(t)$ is not integrable for $0 \leq t < t'$.
\end{proof}

\begin{definition}\label{def:interpolatedphi}
Let $\{\phi_w\}_{w \in \T}$ be an admissible family of Orlicz functions. A real number $t$ such that $I_z(\{\phi_w\})$ is concave and stricly increasing on $[0, t)$ will be called a \emph{witness} of the family.

Fixed a witness $t$ of $\{\phi_w\}_{w \in \T}$ and $z \in \D$, we denote by $\phi_z$ an extension of $(I_z(\{\phi_w\})|_{[0, t)})^{-1}$ to an Orlicz function.
\end{definition}

\begin{lemma}\label{lem:dualfamilyOrlicz}
Let $\{\phi_w\}_{w \in \T}$ be a strongly admissible family of N-functions. Then the family $\{\phi_w^*\}_{w \in \T}$ is admissible.
\end{lemma}
\begin{proof}
By Theorem \ref{thm-duality} and \eqref{eq:equivalencedual} we know that the family $\{\ell_{\phi_w^*}\}_{w \in \T}$ is admissible, so that we have (3) of Definition \ref{def:admissibleOrliczfamily}.

Let $\{q_n\}$ be an enumeration of the nonnegative rationals, and for each $n \in \N$, $w \in \T$ and $y > 0$ let $\psi_{y, n}(w) = q_n y - \phi_w(q_n)$. Then, by the continuity of $\phi_w$, $\phi_w^*(y) = \sup_{n \in \N} \psi_{y, n}(w)$, so that the map $w \mapsto \phi_w^*(y)$ is Borel for each $y \geq 0$, and we have (1) of Definition \ref{def:admissibleOrliczfamily}.

Finally, by \cite[Proposition 2.1.1(ii)]{RaoRen}, $t < \phi_w^{-1}(t) (\phi_w^*)^{-1}(t) < 2t$ for each $w \in \T$ and $t \geq 0$, so that if $t$ is a witness of $\{\phi_w\}_{w \in \T}$ and $0 < t' < t$
\[
\int_{-\pi}^{\pi} \log (\phi_w^*)^{-1}(t') dw \geq 2\pi t' - \int_{-\pi}^{\pi} \log \phi_w^{-1}(t') dw > -\infty
\]
and similarly the integral is not $+\infty$, so that the family $\{\phi_w^*\}$ is admissible.
\end{proof}

\begin{theorem}\label{thm:interpolationfamilies}
Let $\{\phi_w\}_{w \in \T}$ be a strongly admissible family of N-functions. Consider the family $\mathcal{H} = \{X_w\}_{w \in \T}$ such that $X_w = \ell_{\phi_w}$ and let $\phi_z$ be as in Definition \ref{def:interpolatedphi}. Then $X_z = \ell_{\phi_z}$ with equivalence of norms, and there is an $a > 0$ such that
\[
\Omega_z(x)(n) = \frac{x(n)}{\pi} \int_{-\pi}^{\pi} \frac{e^{it}}{(e^{it} - z)^2} \log \Big(\phi_w^{-1}\Big(\phi_z\Big(\frac{a\left|x(n)\right|}{\|x\|_{\ell_{\phi_z}}}\Big)\Big)\Big) dt
\]
for every $x \in \ell_{\phi_z}$.
\end{theorem}
\begin{proof}
Fix a witness $t$ for the family $\{\phi_w\}$. Since we are dealing with sequence spaces, we may find $a \leq 1$ so that if $x \in \ell_{\phi_z}$ is such that $\|x\|_{\ell_{\phi_z}} \leq a$ then $\phi_z(\left|x(n)\right|) < t$ for every $n \in \N$. Indeed, take $a$ such that $\phi_z(a \phi_z^{-1}(1)) < t$. If $\|x\|_{\ell_z} \leq a$ then $\phi_z(\left|x(n)\right|/a) \leq 1$ implies $\phi_z(\left|x(n)\right|) \leq \phi_z(a \phi_z^{-1}(1)) < t$.

Let $\phi(n, w) = \phi_w^{-1}(\phi_z(\left|x(n)\right|))$. Then
\begin{eqnarray*}
&& \exp(\frac{1}{2\pi} \int_{-\pi}^{\pi} P(r, \theta - t) \log \phi(n, e^{it}) dt) \\ 
& = & \exp(\frac{1}{2\pi} \int_{-\pi}^{\pi} P(r, \theta - t) \log \phi_{e^{it}}^{-1}(\phi_z(\left|x(n)\right|)) dt) \\
    & = & I_z(\{\phi_w\}) (\phi_z(\left|x(n)\right|)) \\
    & = & \left|x(n)\right|
\end{eqnarray*}

Also, $\|\phi\|_{\E} = \esssup{\|\phi_w^{-1}(\phi_z(\left|x(n)\right|))\|_{\ell_{\phi_w}}} \leq 1$. Indeed, for $w \in \T$
\begin{eqnarray*}
\sum_{n=1}^{\infty} \phi_w (\phi_w^{-1}(\phi_z(\left|x(n)\right|))) & = & \sum_{n=1}^{\infty} \phi_z(\left|x(n)\right|) \leq 1
\end{eqnarray*}
since $\|x\|_{\ell_z} \leq a \leq 1$. So we obtain $\|x\|_{X_z} \leq \frac{1}{a} \|x\|_{\ell_{\phi_z}}$.

Now let $x \in X_z$ of norm $1$. Consider the family $\mathcal{H}^* = \{Y_w\}_{w \in \T} = \{\ell_{\psi_w}\}_{w \in \T}$, where $\psi_w = \phi_w^*$. By Lemma \ref{lem:dualfamilyOrlicz} the family $\{\psi_w\}$ is admissible and by the previous calculation, $\ell_{\psi_z} \subset Y_z$ with continuous inclusion.

The proof now follows \cite[Lemma 6.3.3]{RaoRen}. By \cite[Proposition 2.1.1(ii)]{RaoRen} we have $t < \phi_z^{-1}(t) (\phi_z^*)^{-1}(t) < 2t$
and $t < \phi_w^{-1}(t) \psi_w^{-1}(t) < 2t$ for every $t > 0$, $w \in \T$. So if $t'$ is a witness for both $\{\phi_w\}$ and $\{\psi_w\}$ then for $t < t'$ that
\begin{eqnarray*}
(\phi_z^*)^{-1}(t) & \leq & \frac{2t}{\phi_z^{-1}(t)} \\
        & = & 2t\exp\Big(-\frac{1}{2\pi} \int_{-\pi}^{\pi} P(z, \theta - s) \log \phi_{e^{is}}^{-1}(t)ds\Big) \\
        & = & 2 \exp\Big(\frac{1}{2\pi} \int_{-\pi}^{\pi} P(z, \theta - s) \log \frac{t}{\phi_{e^{is}}^{-1}(t)} ds\Big) \\
        & < & 2 \exp\Big(\frac{1}{2\pi} \int_{-\pi}^{\pi} P(z, \theta - s) \log \psi_{e^{is}}^{-1} (t) ds\Big) \\
        & = & 2 \psi_z^{-1}(t)
\end{eqnarray*}

If $t_0$ is such that $\psi_z(t) < t'$ for $t \in [0, t_0]$, replacing $t$ by $\psi_z(t)$ we obtain $\psi_z(t) < \phi_z^*(2t)$ for every $t \in [0, t_0]$. This means that there is a constant $C \geq 1$ such that $\|.\|_{\psi_z} \leq C \|.\|_{\phi_z^*}$. So we find a constant $C'$ such that if $x \in X_z$
\begin{eqnarray*}
\sup_{y^* \in B_{(\ell_{\phi_z})^*}} \left|y^*(x)\right| & \leq & \sup_{y^* \in B_{\ell_{\phi_z^*}}} \left|y^*(x)\right| \\
                      & \leq & C \sup_{y^* \in B_{\ell_{\psi_z}}} \left|y^*(x)\right| \\
                      & \leq & C' \sup_{y^* \in B_{X_z^*}} \left|y^*(x)\right| \\
                      & \leq & C' \|x\|_{X_z}
\end{eqnarray*}
(The first inequality is \eqref{eq:equivalencedual})

Now \cite[Proposition 3.3.4]{RaoRen} tells us that we have a continuous inclusion $X_z \subset \ell_{\phi_z}$, and that
\[
B_z(x)(n, w) = a \|x\|_{\ell_{\phi_z}} \phi_w^{-1}(\phi_z\Big(\frac{a\left|x(n)\right|}{\|x\|_{\ell_{\phi_z}}}\Big)) \text{, } x \in \ell_{\phi_z}
\]
is a factorization map. Since $\int_{-\pi}^{\pi} \frac{e^{it}}{(e^{it} - z)^2} dt = 0$, we obtain the formula for $\Omega_z$.
\end{proof}

Of course, the previous theorem generalizes the classical results regarding complex interpolation of couples of Orlicz spaces \cite{Gustavsson1977}, stating that $(\ell_{\phi_0}, \ell_{\phi_1})_{\theta} = \ell_{\phi}$, with
\[
\phi^{-1} = (\phi_0^{-1})^{1 - \theta} (\phi_1^{-1})^{\theta}
\]

See \cite{IntepolationOrlicz} for a generalization to complex interpolation of couples of Orlicz spaces with respect to a vector measure.

\section{The example}\label{sec:example}

We have the following characterization of the interpolation spaces $X_z$ for finite families:

\begin{theorem}\cite[Proposition 3.11]{Stability}\label{thm:interpolationfinitefamilies}
Let $\mathcal{H} = \{X^j; A_j\}_{j=1}^n$ be a finite admissible family. Let $\mu_z$ denote harmonic measure on $\T$ with respect to $z \in \D$. Then $X_z$ coincides with the space
\begin{equation}\label{eq:def-prod}
\prod_{j=1}^n (X^j)^{\mu_z(A_j)} = \{x \in \ell_0 : \left|x\right| = \prod_{j=1}^n \left|x_j\right|^{\mu_z(A_j)}, f_j \in X^j\}
\end{equation}
and $\|x\|_{X_z} = \inf\{\prod_{j=1}^n \|x_j\|_{X^j}^{\mu_z(A_j)}\}$, where the infimum is over all $x_j$ as in \eqref{eq:def-prod}.
\end{theorem}

In the proof of Theorem \ref{thm:interpolationfamilies} we asked that the functions $\phi_w$ be N-functions satisfying the $\Delta_2$-condition at 0 so that we may use duality and the invertibility of $\phi_w^*$. The next theorem shows that when we are dealing with finite families we do not need these tools. Its proof should be compared with that of \cite[Lemma 4.1]{IntepolationOrlicz}.

\begin{theorem}\label{thm:interpolationfiniteOrlicz}
Let $\mathcal{H} = \{\ell_{\phi_j}; A_j\}_{j=1}^n$ be an admissible family such that the functions $\phi_j$ are non-degenerate. Then $X_z = \ell_{\phi_z}$ with equivalence of norms for every $z \in \D$, where
\[
\phi_z^{-1} = \prod_{j = 1}^n (\phi_j^{-1})^{\mu_z(A_j)}
\]

The induced centralizer is
\[
\Omega_z(x) = x \sum_{j=1}^n \psi_j'(z) \log \phi_j^{-1}\Big(\phi_z\Big(\frac{\left|x(n)\right|}{\|x\|_{\ell_{\phi_z}}}\Big)\Big)
\]
for $x \in \ell_{\phi_z}$, where $\psi_j$ is an analytic function on $\D$ which real part agrees with $\chi_{A_j}$ on $\T$.
\end{theorem}
\begin{proof}
We begin by noticing that the definition of $\phi_z$ given above coincides with that of Definition \ref{def:interpolatedphi} when the witness is $+\infty$. If we mimic now the first part of the proof of Theorem \ref{thm:interpolationfamilies} for $a=1$ we obtain $\ell_{\phi_z} \subset X_z$ and  $\|.\|_{X_z} \leq \|.\|_{\ell_{\phi_z}}$.

For the other inclusion, let $x \in X_z$ of norm 1. By Theorem \ref{thm:interpolationfinitefamilies}, we may write
\[
\left|x\right| = \prod_{j=1}^n \left|x_j\right|^{\mu_z(A_j)}
\]
with $\|x_j\|_{\ell_{\phi_j}} \in [1, 2]$. Let $h_j \in \C^{\N}$ be given by $h_j(n) = \phi_j\Big(\frac{\left|x_j(n)\right|}{2}\Big)$, and let $h = \sum_{j=1}^n h_j$. Then
\begin{eqnarray*}
\frac{\left|x\right|}{2} & = & \prod \Big(\frac{\left|x_j\right|}{2}\Big)^{\mu_z(A_j)} \\
               & = & \prod \phi_j^{-1}(h_j)^{\mu_z(A_j)} \\
               & \leq & \prod \phi_j^{-1}(h)^{\mu_z(A_j)} \\
               & = & \phi_z^{-1}(h)
\end{eqnarray*}
so that
\[
\sum\limits_k \phi_z\Big(\frac{\left|x(k)\right|}{2}\Big) \leq \sum\limits_k h(k) = \sum\limits_{k, j} \phi_j\Big(\frac{\left|x_j(k)\right|}{2}\Big) \leq n + 1
\]
That is, $x \in \ell_{\phi_z}$, and therefore $X_z = \ell_{\phi_z}$ as sets. By the first part of the proof and the open mapping theorem, the norms are equivalent. We obtain a factorization map and $\Omega_z$ as in Theorem \ref{thm:interpolationfamilies}.
\end{proof}

We now present the promised example:

\begin{theorem}\label{thm:example}
Let $A_0 = \{e^{it} : 0 \leq t < \frac{2\pi}{3}\}$, $A_1 = \{e^{it} : \frac{2\pi}{3} \leq t < \frac{4\pi}{3}\}$ and $A_2 = \{e^{it} : \frac{4\pi}{3} \leq t < 2\pi\}$. Let $\mathcal{H} = \{X_w\}_{w \in \T}$, with $X_w = \ell_{\phi^j}$ for $w \in A_j$, where
\[
\phi^0(t) = t
\]
\[
\phi^1(t) = 5^{-2} t^2 \left|\log(t)\right|^4
\]
\[
\phi^2(t) = 5^{-2} e^{-2 + 2 \sqrt{1 - t}}(2t + 2 \sqrt{1 - t} - 2)^2
\]
on a neighborhood of $0$. Then $X_0 = \ell_2$ with equivalence of norms, and the induced centralizer $\Omega_0$ is not projectively equivalent to any centralizer obtained by complex interpolation from two K\"othe function spaces.
\end{theorem}

We may check that $\phi^1$ and $\phi^2$ are Orlicz functions on a neighborhood of $0$ by differentiation. One may also check that both satisfy the $\Delta_2$ condition at 0. We leave the bothersome details to the reader.

Recall that two Orlicz functions $\phi$ and $\psi$ satisfying the $\Delta_2$-condition are \emph{equivalent}, written $\phi \sim \psi$, if there are $K > 0$, $t_0 > 0$ such that for every $0 \leq t \leq t_0$
\begin{equation}\label{eq:equiv}
K^{-1} \leq \frac{\phi(t)}{\psi(t)} \leq K
\end{equation}
Equivalently,
\[
\phi^{-1}(K^{-1}\psi(t)) \leq t \leq \phi^{-1}(K\psi(t))
\]
for $0 \leq t \leq t_0$, which happens if and only if $\ell_{\phi} = \ell_{\psi}$ with equivalence of norms.

\begin{props}
Under the conditions of Theorem \ref{thm:example}, $X_0 = \ell_2$ with equivalence of norms.
\end{props}
\begin{proof}
$X_0 = \ell_2$ with equivalence of norms if and only if there are $K > 0$ and $t_0 > 0$ such that for every $0 < t \leq t_0$ 
\begin{eqnarray*}
&& ((\phi^0)^{-1}(K^{-1}t^2))^{\frac{1}{3}}((\phi^1)^{-1}(K^{-1}t^2))^{\frac{1}{3}}((\phi^2)^{-1}(K^{-1}t^2))^{\frac{1}{3}} \\
&& \leq t \\
&&\leq ((\phi^0)^{-1}(Kt^2))^{\frac{1}{3}}((\phi^1)^{-1}(Kt^2))^{\frac{1}{3}}((\phi^2)^{-1}(Kt^2))^{\frac{1}{3}} 
\end{eqnarray*}

Since $\phi^0(t) = t$,
\begin{eqnarray*}
&& K^{-\frac{1}{3}}t^{\frac{2}{3}}((\phi^1)^{-1}(K^{-1}t^2))^{\frac{1}{3}}((\phi^2)^{-1}(K^{-1}t^2))^{\frac{1}{3}} \\
&& \leq t \\
&& \leq K^{-\frac{1}{3}}t^{\frac{2}{3}}((\phi^1)^{-1}(Kt^2))^{\frac{1}{3}}((\phi^2)^{-1}(Kt^2))^{\frac{1}{3}} 
\end{eqnarray*}
that is
\[
((\phi^1)^{-1}(K^{-1}t^2))^{\frac{1}{3}}((\phi^2)^{-1}(K^{-1}t^2))^{\frac{1}{3}} \leq K^{\frac{1}{3}} t^{\frac{1}{3}} \leq ((\phi^1)^{-1}(Kt^2))^{\frac{1}{3}}((\phi^2)^{-1}(Kt^2))^{\frac{1}{3}} 
\]
so
\[
((\phi^1)^{-1}(K^{-1}t^2))^{\frac{1}{2}}((\phi^2)^{-1}(K^{-1}t^2))^{\frac{1}{2}} \leq K^{\frac{1}{2}} t^{\frac{1}{2}} \leq ((\phi^1)^{-1}(Kt^2))^{\frac{1}{2}}((\phi^2)^{-1}(Kt^2))^{\frac{1}{2}} 
\]
Taking $u = K^{\frac{1}{2}} t^\frac{1}{2}$
\begin{equation}\label{eq:musthave}
((\phi^1)^{-1}(K^{-2}u^4))^{\frac{1}{2}}((\phi^2)^{-1}(K^{-2}u^4))^{\frac{1}{2}} \leq u \leq ((\phi^1)^{-1}(u^4))^{\frac{1}{2}}((\phi^2)^{-1}(u^4))^{\frac{1}{2}} 
\end{equation}

We must show that \eqref{eq:musthave} is satisfied. Take $\psi^1(t) = t \left|\log(t)\right|^2$. From now on we will use the notation for the inverse of a function when the function is injective on a neighborhood of 0. Define $\psi^2 = (\psi^1)^*$ as if $\psi^1$ was an Orlicz function. So
\[
\psi^2(s) = \max_{t \geq 0} (ts - \psi^1(t)) = \max_{t \geq 0} t(s - \log(t)^2)
\]

Let $f_s(t) = t(s - \log(t)^2)$. Then $f_s'(t) = s - \log(t)^2 - 2\log(t)$, which is 0 if and only if $t = e^{-1 \pm \sqrt{1 - s}}$ for $0 \leq s < 1$. Then
\[
\psi^2(s) = e^{-1 + \sqrt{1 - s}} (s - (-1 + \sqrt{1 - s})^2)
\]
for $0 \leq s < 1$. This function is increasing on a neighborhood of 0. By definition $ts \leq \psi^1(t) + \psi^2(s)$, so
\[
(\psi^1)^{-1}(t)(\psi^2)^{-1}(t) \leq 2t
\]

Let $\xi^i = (\psi^i)^2$. Then $(\xi^i)^{-1}(t) = (\psi^i)^{-1}(t^{\frac{1}{2}})$ and $(\xi^1)^{-1}(t) (\xi^2)^{-1}(t) \leq 2 t^{\frac{1}{2}}$. It follows that $(\xi^1)^{-1}(t^4)^{\frac{1}{2}} (\xi^2)^{-1}(t^4)^{\frac{1}{2}} \leq \sqrt{2} t$. Write $t = \frac{u}{\sqrt{10}}$, so that
\[
(\xi^1)^{-1}(\frac{u^4}{100})^{\frac{1}{2}} (\xi^2)^{-1}(\frac{u^4}{100})^{\frac{1}{2}} \leq \frac{u}{\sqrt{5}}
\]

For $t$ near 1 from below, the function $f(t) = \psi^2(\frac{\psi^1(t)}{5t}) - \psi^1(t)$ is negative. This implies, taking $\psi^1(t) = u$, that $\frac{u}{5} \leq (\psi^1)^{-1}(u) (\psi^2)^{-1}(u)$ for $u$ small enough. In particular $\frac{u}{\sqrt{5}} \leq (\xi^1)^{-1}(u^4)^{\frac{1}{2}}(\xi^2)^{-1}(u^4)^{\frac{1}{2}}$. So
\[
(\xi^1)^{-1}(\frac{u^4}{100})^{\frac{1}{2}} (\xi^2)^{-1}(\frac{u^4}{100})^{\frac{1}{2}} \leq \frac{u}{\sqrt{5}} \leq (\xi^1)^{-1}(u^4)^{\frac{1}{2}}(\xi^2)^{-1}(u^4)^{\frac{1}{2}}
\]

Noticing that $\phi^i = 5^{-2}\xi^i$ and writing $\frac{u}{\sqrt{5}} = t$, we obtain
\begin{equation}\label{eq:inequalities}
(\phi^1)^{-1}(\frac{t^4}{100})^{\frac{1}{2}} (\phi^2)^{-1}(\frac{t^4}{100})^{\frac{1}{2}} \leq t \leq (\phi^1)^{-1}(t^4)^{\frac{1}{2}}(\phi^2)^{-1}(t^4)^{\frac{1}{2}}
\end{equation}
\end{proof}

Let $\Omega_0 \sim \Omega_1 + i \Omega_2$ with $\Omega_1$ and $\Omega_2$ real. Kalton shows in \cite[Lemma 7.1]{KaltonKothe} that we may take
\begin{equation}\label{eq:Omega1def}
\Omega_1(x) = Re(\Omega_0(Re(x))) + i Re(\Omega_0(Im(x)))
\end{equation}
and
\begin{equation}\label{eq:Omega2def}
\Omega_2(x) = Im(\Omega_0(Re(x))) + i Im(\Omega_0(Im(x)))
\end{equation}

\begin{lemma}\label{lem:AnBn}
For $n \in \N$ let $s_n \in \ell_2$ be the unit vectors given by $s_n = \frac{1}{\sqrt{n}} \sum_{j=1}^n e_j$. Under the conditions of Theorem \ref{thm:example}, if we write $\Omega_0 \sim$ $\Omega_1 + i\Omega_2$ with $\Omega_1$ and $\Omega_2$ as in \eqref{eq:Omega1def} and \eqref{eq:Omega2def} then $\Omega_1(s_n) = \frac{\sqrt{3}}{2\pi} A_n s_n$ and $\Omega_2(s_n) = \frac{3}{2\pi} B_n s_n$, where $A_n = \log\Big(\frac{(\phi^0)^{-1}(1/\sqrt{n}) (\phi^2)^{-1}(1/\sqrt{n})}{((\phi^1)^{-1}(1/\sqrt{n}))^2}\Big)$ and $B_n = \log\Big( \frac{(\phi^2)^{-1}(1/\sqrt{n})}{(\phi^0)^{-1}(1/\sqrt{n})}\Big)$.
\end{lemma}
\begin{proof}
By the proofs of Theorems \ref{thm:interpolationfamilies} and \ref{thm:interpolationfiniteOrlicz}, we may take a factorization map $B_0$ satisfying $B_0(s_n)(j, w) = \phi_w^{-1}(\phi_z(\frac{1}{\sqrt{n}}))$ for $1 \leq j \leq n$, and 0 otherwise.

Then, for $1 \leq j \leq n$,
\begin{eqnarray*}
\Omega_0(s_n)(j) & = & \frac{s_n(j)}{\pi} \int_{-\pi}^{\pi} \frac{1}{e^{it}} \log B_0(s_n)(j, e^{it}) dt \\
        & = & \frac{1}{{2\pi}\sqrt{n}} ((\sqrt{3} - 3i) \log (\phi^0)^{-1}(\frac{1}{\sqrt{n}}) - 2\sqrt{3} \log (\phi^1)^{-1}(\frac{1}{\sqrt{n}}) \\
         && + (\sqrt{3} + 3i) \log (\phi^2)^{-1}(\frac{1}{\sqrt{n}})) \\
        & = & \frac{1}{2\pi\sqrt{n}}(\sqrt{3} \log\Big(\frac{(\phi^0)^{-1}(\frac{1}{\sqrt{n}}) (\phi^2)^{-1}(\frac{1}{\sqrt{n}})}{((\phi^1)^{-1}(\frac{1}{\sqrt{n}}))^2}\Big) + 3 i \log\Big( \frac{(\phi^2)^{-1}(\frac{1}{\sqrt{n}})}{(\phi^0)^{-1}(\frac{1}{\sqrt{n}})}\Big))
\end{eqnarray*}
and 0 otherwise. The result now follows from \eqref{eq:Omega1def} and \eqref{eq:Omega2def}.
\end{proof}

Since $X_0 = \ell_2$ with equivalence of norms, there are $k, K > 0$ such that
\[
\frac{k}{n} \leq \phi_0(\frac{1}{\sqrt{n}}) \leq \frac{K}{n}
\]

\begin{lemma}\label{lem:almostthere}
\begin{enumerate}
    \item If $\Omega_1$ is trivial then there is $C > 0$ such that $\left|A_n\right| \leq C$ for every $n$.
    \item If there is $\gamma$ such that $\Omega_1$ is equivalent to $\gamma \Omega_2$ then there is $C > 0$ such that $\left|\gamma A_n - B_n\right| \leq C$ for every $n$.
\end{enumerate}
\end{lemma}
\begin{proof}
(1) If $\Omega_1$ is trivial there is a sequence $f \in \C^{\N}$ such that $\Omega_1 - M_{f}$ is bounded.

For a given Orlicz function space $\ell_{\phi}$ the norm of $e_n$ is independent of $n$. This implies that we may choose $B_0$ in such a way that $\Omega_0(e_n) = a e_n$ for some $a \in \C$ and every $n$. Since $\Omega_1 - M_{f}$ is bounded, $f \in \ell_{\infty}$.

So, for every $n$,
\[
\|\Omega_0(s_n)\|_2 \leq \|\Omega_0(s_n) - f(s_n)\|_2 + \|f(s_n)\|_2 \leq C\|s_n\|_2
\]

The result follows from Lemma \ref{lem:AnBn}. (2) is proved similarly.
\end{proof}

Theorem \ref{thm:example} follows from Lemma \ref{lem:realcent}, Lemma \ref{lem:almostthere} and the next two results.

\begin{lemma}
The sequence $(\left|A_n\right|)_n$ of Lemma \ref{lem:AnBn} is not bounded.
\end{lemma}
\begin{proof}
Let $u_n^4 = \frac{1}{\sqrt{n}}$. We have $A_n = \log\Big(\frac{(\phi^0)^{-1}(u_n^4) (\phi^2)^{-1}( u_n^4)}{((\phi^1)^{-1}( u_n^4))^2}\Big)$. By \eqref{eq:inequalities}
\[
A_n \geq \log\Big(\frac{u_n^4 u_n^2}{((\phi^1)^{-1}(u_n^4))^3}\Big)
\]

The right side is bounded from above if and only if there is $C > 0$ such that $\frac{u_n^6}{((\phi^1)^{-1}(u_n^4))^3} \leq e^C$ for $n$ big enough, that is
\[
\phi^1(\frac{u_n^2}{e^{\frac{C}{3}}}) \leq u_n^4
\]
which cannot happen.
\end{proof}

\begin{lemma}
Let $A_n$ and $B_n$ be as in Lemma \ref{lem:AnBn}. For every $\gamma \in \mathbb{R}$ the sequence $(\left|\gamma A_n - B_n\right|)_n$ is not bounded.
\end{lemma}
\begin{proof}
We must consider a number of cases.

\textit{Case 1: $\gamma = 1$}

Let $r_n = \frac{1}{\sqrt{n}}$. Then $A_n - B_n = \log\Big(\frac{r_n^2}{((\phi^1)^{-1}(r_n))^2}\Big)$. If $\left|A_n - B_n\right|$ is bounded there is $C$ such that for $n$ big enough $\phi^1(\frac{r_n}{C}) \geq r_n$. For $n$ big enough this is the same as
\[
5^{-2} \frac{r_n}{C^2} \left|\log \frac{r_n}{C}\right|^4 \geq 1
\]

\textit{Case 2: $\gamma - 1 > 0$}

Let $\frac{u_n^4}{100} = \frac{1}{\sqrt{n}}$. Then by \eqref{eq:inequalities}
\[
\gamma A_n - B_n \leq \log\Big(\frac{u_n^{4(\gamma + 1)} u_n^{2(\gamma - 1)}}{100^{\gamma + 1} ((\phi^1)^{-1}(\frac{u_n^4}{100}))^{3\gamma - 1}}\Big)
\]

The right side is bounded from below if and only if for some $C > 0$ we have $\frac{u_n^4}{100} \leq \phi^1\Big(\frac{u_n^{\frac{2(3\gamma + 1)}{3\gamma - 1}}}{C}\Big)$. For $n$ big enough this is the same as
\[
\frac{u_n^{\frac{-8}{3\gamma - 1}}}{100} \leq 5^{-2} \frac{1}{C^2} \left| \log \frac{u_n^{\frac{2(3\gamma + 1)}{3\gamma - 1}}}{C} \right|^4
\]
which cannot happen for $n$ big enough.

The case $\gamma < 1$ must be divided in a number of subcases. Since they are all very similar, we will do one of them as illustration.

\textit{Case 3: $\gamma - 1 < 0$ and $3\gamma - 1 > 0$}

Let $u_n^4 = \frac{1}{\sqrt{n}}$. By \eqref{eq:inequalities}
\[
\gamma A_n - B_n \leq \log\Big(\frac{u_n^{4(\gamma + 1)} u_n^{2(\gamma - 1)}}{((\phi^1)^{-1}(u_n^4))^{3\gamma - 1}}\Big)
\]
Let us show that the sequence on the right has $-\infty$ as limit. Were it bounded from below we should have for some $C > 0$ that $u_n^4 \leq \phi^1(\frac{u_n^{\frac{2(3\gamma + 1)}{3\gamma - 1}}}{C})$. For $n$ big enough this is
\[
u_n^{\frac{-8}{3\gamma - 1}} \leq \frac{5^{-2}}{C^2} \left| \log \frac{u_n^{\frac{2(3\gamma + 1)}{3\gamma - 1}}}{C} \right|^4 
\]
which cannot happen when $n$ grows.

The other cases one should check are 4.\ $\gamma = \frac{1}{3}$; 5.\ $3\gamma - 1 < 0$ and $3\gamma + 1 > 0$; 6.\ $\gamma = -\frac{1}{3}$; and 7.\ $\gamma  < -\frac{1}{3}$.

\end{proof}

\section{Acknowledgements}

We would like to thank F\'elix Cabello S\'anchez, Jes\'us Castillo and Valentin Ferenczi for discussions regarding this work, and Micha\l\ Wojciechowski for suggesting the study of interpolation of families of Orlicz spaces.

\bibliographystyle{amsplain}
\bibliography{refs}

\end{document}